\documentclass[a4paper,11pt]{amsart}
\usepackage{mathrsfs}
\usepackage{cases}
\usepackage{amsfonts}
\usepackage{graphicx}
\usepackage{amsmath}
\usepackage{amssymb}
\usepackage[all]{xypic}
\usepackage[all]{xy}
\usepackage{color}
\usepackage{colordvi}
\usepackage{multicol}
\begin{document}
\input xy
\xyoption{all}

\newcommand{\ind}{\operatorname{inj.dim}\nolimits}
\newcommand{\id}{\operatorname{id}\nolimits}
\newcommand{\Mod}{\operatorname{Mod}\nolimits}
\newcommand{\End}{\operatorname{End}\nolimits}
\newcommand{\Ext}{\operatorname{Ext}\nolimits}
\newcommand{\Hom}{\operatorname{Hom}\nolimits}
\newcommand{\aut}{\operatorname{Aut}\nolimits}
\newcommand{\Ker}{{\operatorname{Ker}\nolimits}}
\newcommand{\Iso}{\operatorname{Iso}\nolimits}
\newcommand{\Coker}{\operatorname{Coker}\nolimits}
\renewcommand{\dim}{\operatorname{dim}\nolimits}
\newcommand{\Cone}{{\operatorname{Cone}\nolimits}}
\renewcommand{\Im}{\operatorname{Im}\nolimits}

\newcommand{\cc}{{\mathcal C}}
\newcommand{\ce}{{\mathcal E}}
\newcommand{\cs}{{\mathcal S}}
\newcommand{\cf}{{\mathcal F}}
\newcommand{\cx}{{\mathcal X}}
\newcommand{\cy}{{\mathcal Y}}
\newcommand{\cl}{{\mathcal L}}
\newcommand{\ct}{{\mathcal T}}
\newcommand{\cu}{{\mathcal U}}
\newcommand{\cm}{{\mathcal M}}
\newcommand{\cv}{{\mathcal V}}
\newcommand{\ch}{{\mathcal H}}
\newcommand{\ca}{{\mathcal A}}
\newcommand{\mcr}{{\mathcal R}}
\newcommand{\cb}{{\mathcal B}}
\newcommand{\ci}{{\mathcal I}}
\newcommand{\cj}{{\mathcal J}}
\newcommand{\cp}{{\mathcal P}}
\newcommand{\cg}{{\mathcal G}}
\newcommand{\cw}{{\mathcal W}}
\newcommand{\co}{{\mathcal O}}
\newcommand{\cd}{{\mathcal D}}
\newcommand{\cn}{{\mathcal N}}
\newcommand{\ck}{{\mathcal K}}
\newcommand{\calr}{{\mathcal R}}
\newcommand{\ol}{\overline}
\newcommand{\ul}{\underline}
\newcommand{\cz}{{\mathcal Z}}
\newcommand{\st}{[1]}
\newcommand{\ow}{\widetilde}
\renewcommand{\P}{\mathbf{P}}
\newcommand{\pic}{\operatorname{Pic}\nolimits}
\newcommand{\Spec}{\operatorname{Spec}\nolimits}
\newtheorem{theorem}{Theorem}[section]
\newtheorem{acknowledgement}[theorem]{Acknowledgement}
\newtheorem{algorithm}[theorem]{Algorithm}
\newtheorem{axiom}[theorem]{Axiom}
\newtheorem{case}[theorem]{Case}
\newtheorem{claim}[theorem]{Claim}
\newtheorem{conclusion}[theorem]{Conclusion}
\newtheorem{condition}[theorem]{Condition}
\newtheorem{conjecture}[theorem]{Conjecture}
\newtheorem{construction}[theorem]{Construction}
\newtheorem{corollary}[theorem]{Corollary}
\newtheorem{criterion}[theorem]{Criterion}
\newtheorem{definition}[theorem]{Definition}
\newtheorem{example}[theorem]{Example}
\newtheorem{exercise}[theorem]{Exercise}
\newtheorem{lemma}[theorem]{Lemma}
\newtheorem{notation}[theorem]{Notation}
\newtheorem{problem}[theorem]{Problem}
\newtheorem{proposition}[theorem]{Proposition}
\newtheorem{remark}[theorem]{Remark}
\newtheorem{solution}[theorem]{Solution}
\newtheorem{summary}[theorem]{Summary}
\newtheorem*{thm}{Theorem}

\def \bp{{\mathbf p}}
\def \bA{{\mathbf A}}
\def \bL{{\mathbf L}}
\def \bF{{\mathbf F}}
\def \bS{{\mathbf S}}
\def \bC{{\mathbf C}}
\def \bD{{\mathbf D}}
\def \Z{{\Bbb Z}}
\def \F{{\Bbb F}}
\def \C{{\Bbb C}}
\def \N{{\Bbb N}}
\def \Q{{\Bbb Q}}
\def \G{{\Bbb G}}
\def \X{{\Bbb X}}
\def \P{{\Bbb P}}
\def \K{{\Bbb K}}
\def \E{{\Bbb E}}
\def \A{{\Bbb A}}
\def \BH{{\Bbb H}}
\def \T{{\Bbb T}}

\title[Modified Ringel-Hall algebras]{Modified Ringel-Hall Algebras, Green's Formula\\ and Derived Hall Algebras }

\thanks{Ji Lin is supported partially by the National Natural Science Foundation of China (Grant No. 11701473), the Natural Science Foundation of Fuyang Normal University (Grant No. 2018FSKJ02ZD)
 and Liangang Peng is supported partially by the National Natural Science Foundation of China (Grant No. 11521061).}
\author[Lin]{Ji Lin}

\address{Department of Mathematics, Sichuan University, Chengdu 610064, P.R.China and
Department of Mathematics and Statistics, Fuyang Normal University, Fuyang 236037, P.R.China}
\email{jlin@fync.edu.cn}
\author[Peng]{Liangang Peng}
\address{Department of Mathematics, Sichuan University, Chengdu 610064, P.R.China}
\email{penglg@scu.edu.cn}


\begin{abstract}
In this paper we define the modified Ringel-Hall algebra $\cm\ch(\ca)$ of a hereditary abelian category $\ca$ from the category $C^b(\mathcal{A})$ of bounded $\mathbb{Z}$-graded complexes. Two main results have been obtained. One is to give a new proof of  Green's formula on Ringel-Hall numbers  by using the associative multiplication of the modified Ringel-Hall algebra. The other is to show that in certain twisted cases the derived Hall algebra can be embedded in the modified Ringel-Hall algebra. As a consequence of the second result, we get that in certain twisted cases the modified Ringel-Hall algebra is isomorphic to the tensor algebra of the derived Hall algebra and the torus  of acyclic complexes and so the modified Ringel-Hall algebra is invariant under derived equivalences.
\end{abstract}

\maketitle
\section{Introduction}

The Ringel-Hall algebra of an associative algebra $A$ was designed by C. M. Ringel in \cite{R0} to realize the positive part of a complex simple Lie algebra. In fact, the situation was much better. When $A$  is hereditary,  Ringel \cite{R2} in finite type case proved a surprising result  that  the Ringel-Hall algebra of $A$ in some twisted case can realize the positive part of the corresponding quantum group. Later J. A. Green \cite{Gr} found a famous formula on Ringel-Hall numbers,  usually called Green's formula, to show the bialgebra structure of the Ringel-Hall algebra and then generalized the Ringel's result to any type case. These works have led to extensive research into Ringel-Hall algebras, and a great deal of progresses has been made, see for example \cite{R5,Rie,GP,X,SV,Kap,PX1,PX2,LP,Sch1,Sch2,Cr,T,XX,KS,XX2,Br,Gor13,Gor16}.

One of remarkable progresses is the To\"{e}n's derived Hall algebra of a locally finite dg-category in \cite{T}, which was generalized by Xiao-Xu in \cite{XX} to any triangulated category with locally homological-finite conditions. In particular, the derived Hall algebra is well-defined for  the bounded
derived category $D^b(A)$ of  any  finitely dimensional algebra $A$.

Recently, inspired by the work of T. Bridgeland in \cite{Br} and of M. Gorsky in \cite{Gor13} on constructing Ringel-Hall algebras from $\Z/2$-graded complexes, the modified Ringel-Hall algebra was defined in \cite{LuP} from $\Z/2$-graded complexes for any hereditary abelian category $\ca$ which may not have enough projective objects, and it was shown that such modified Ringel-Hall algebra is isomorphic to the corresponding Drinfeld double Ringel-Hall algebra.

In this paper, similar to \cite{LuP} we define the modified Ringel-Hall algebra from $\Z$-graded bounded complexes for any hereditary abelian category $\ca$. First of all, we find that the associative multiplication of the modified Ringel-Hall algebra implies  Green's formula. This provides a new proof of  Green's formula. Our proof is completely different from  the original one and  easier to read. Secondly we prove that in certain twisted cases the derived Hall algebra can be embedded in the modified Ringel-Hall algebra. As a consequence, we get that in certain twisted cases the modified Ringel-Hall algebra is isomorphic to the tensor algebra of the derived Hall algebra and the torus  of acyclic complexes and so the modified Ringel-Hall algebra is invariant under derived equivalences. The second result and its consequences above  are closely related to Gorsky's results in \cite{Gor16} in which he defined the semi-derived Hall algebra from any Frobenius category $\mathcal{F}$ satisfying certain finiteness conditions and obtained similar relations between his semi-derived Hall algebra and the derived Hall algebra of the stable category  $\underline{\mathcal{F}}$.

The paper is organized as follows. In Section 2 we obtain locally homological finiteness of bounded complexes and calculate the Euler forms for some special bounded complexes. In Section 3,  we define the modified Ringel-Hall algebra  of a hereditary abelian category $\ca$ from the category $C^b(\mathcal{A})$ of bounded ($\mathbb{Z}$-graded) complexes and study  its structure. In particular we find a basis of modified Ringel-Hall algebra and describe the modified Ringel-Hall algebra by the generators and relations. In Section 4, we give a new proof of Green's formula. Finally we recall the definition of the derived Hall algebra and prove our embedding theorem and its consequences in Section 5.

\section{Locally Homological Finiteness of Bounded Complexes}

Unless specified, throughout this paper $k$ denotes a finite field and $\ca$ is an essentially small hereditary abelian  $k$-category which is finitary, i.e.,
$$\dim_k\Hom_{\ca}(M,N)<\infty,\,\,\, \dim_{k}\Ext^1_{\ca}(M,N)<\infty,\,\,\, \forall M,N\in\ca.$$

Given an object $X\in\ca$, and $m\in$$\mathbb{Z}$, denote by $K_{X,m}$ the acyclic complex
$$\cdots \rightarrow0 \rightarrow X \xrightarrow{1_X} X \rightarrow 0\rightarrow\cdots,$$
where $X$ sits in the degrees $m-1$ and $m$.

\begin{lemma}\label{lemma of extension equals 0 with length 2}
For any $X\in\ca$ and $m\in\mathbb{Z}$, the Ext-dimension of $K_{X,m}$ is less than or equal to 1, i.e., $\Ext_{\cc^b(\ca)}^p(K_{X,m}, M)=0$ and $\Ext_{\cc^b(\ca)}^p(M, K_{X,m})=0$ for any $M\in\cc^b(\ca)$ and $p\geq 2$.
\end{lemma}

\begin{proof}
For any exact sequence
$$\xi: 0\rightarrow M\rightarrow V_{p}\rightarrow V_{p-1}\rightarrow\cdots\rightarrow V_2\rightarrow V_1\rightarrow K_{X, m}\rightarrow 0$$
in $\Ext^p_{\cc^b(\ca)}(K_{X,m}, M)$, one can easily get the following commutative diagram of the exact sequences in $\cc^b(\ca)$
\[\xymatrix{
\xi': 0\ar[r]&K_{M^{m-1},m}\ar[r]\ar[d]_f&K_{V^{m-1}_{p},m}\ar[r]\ar[d]&\cdots\ar[r]&K_{V^{m-1}_{1},m}\ar[r]\ar[d]&K_{X, m}\ar[r]\ar@{=}[d]&0\\
\xi:\, 0\ar[r]&M\ar[r]&V_{p}\ar[r]&\cdots\ar[r]&V_{1}\ar[r]&K_{X, m\ar[r]}&0,
}\]
where the exact sequence $\xi'$ is naturally determined by the $(m-1)$-th component exact sequence of $\xi$. Since $\ca$ hereditary and $p\geq 2$, it is
clear that $[\xi']=0$, where $[\xi']$ is the equivalent class of the exact sequence $\xi'$. One can easily see that $\Ext^p_{\cc^b(\ca)}(K_{X, m}, f)([\xi'])=[\xi]$ and so $[\xi]=0$. This shows that $\Ext^p_{\cc^b(\ca)}(K_{X,m}, M)=0$. By duality one can prove that $\Ext_{\cc^b(\ca)}^p(M, K_{X,m})=0$ for any $M\in\cc^b(\ca)$ and $p\geq 2$.
\end{proof}

\begin{proposition}\label{proposition extension 2 zero}
For any $M,K\in\cc^b(\ca)$ with $K$  acyclic, we have
$$\Ext^{p}_{\cc^b(\ca)}(K, M)=0,\quad \Ext^{p}_{\cc^b(\ca)}(M,K)=0, \mbox{ for any } p\geq2.$$
\end{proposition}
\begin{proof}
Note that the bounded acyclic complex $K$ can be obtained by finitely repeated extensions of some special acyclic complexes of form $K_{X,m}$. So from the above lemma we get the results as required.
\end{proof}

For any object $A\in\ca$, let $U_{A,m}$ be the stalk complex with $A$ concentrated in the degree $m$.
\begin{proposition}\label{proposition homological finiteness of stalk complex}
For any $A,B\in\ca$ and $m>n\in \mathbb{Z}$, we have

$(1) \Ext^p_{\cc^b(\ca)}(U_{A,m},U_{B,m})=0$ for any $p\geq 2;$

$(2) \Ext^p_{\cc^b(\ca)}(U_{A,m},U_{B,n})=0$ for any $p\geq 0;$

$(3)$ $\Ext^p_{\cc^b(\ca)}(U_{B,n},U_{A,m})\cong\begin{cases}\Hom_{\ca}(B,A)&p=m-n,\\\Ext^1_{\ca}(B,A)&p=m-n+1,\\0&otherwise.\end{cases}$
\end{proposition}
\begin{proof}
(1) It is not hard to see that $\Ext^p_{\cc^b(\ca)}(U_{A,m},U_{B,m})$ is isomorphic to $\Ext^p_{\ca}(A,B)$ for any $p\geq 0$. So from  $\ca$ hereditary we have $ \Ext^p_{\cc^b(\ca)}(U_{A,m},U_{B,m})=0$ for any $p\geq 2;$

(2) Clearly one has the short exact sequence $0\rightarrow U_{A,m}\rightarrow K_{A,m}\rightarrow U_{A,m-1}\rightarrow0$. For any $p\geq 2$, using $\Hom_{\cc^b(\ca)}(-, U_{B,n})$ acting on this short exact sequence, and by Lemma \ref{lemma of extension equals 0 with length 2} we have an isomorphism between $\Ext^p_{\cc^b(\ca)}(U_{A,m},U_{B,n})$ and $\Ext^{p+1}_{\cc^b(\ca)}(U_{A,m-1},U_{B,n})$ and so $\Ext^p_{\cc^b(\ca)}(U_{A,m},U_{B,n})$ $\cong\Ext^{p+m-n}_{\cc^b(\ca)}(U_{A,n},U_{B,n})=0$. In addition it is easy to see that $\Ext^1_{\cc^b(\ca)}(U_{A,m},U_{B,n})=0$ and $\Hom_{\cc^b(\ca)}(U_{A,m},U_{B,n})=0$.

(3) It is not hard to see that $\Hom_{\cc^b(\ca)}(U_{B,n},K_{A,m})=0$ and $\Ext^1_{\cc^b(\ca)}(U_{B,n},K_{A,m})=0$. Using $\Hom_{\cc^b(\ca)}(U_{B,n}, -)$ acting on the short exact sequence $0\rightarrow U_{A,m}\rightarrow K_{A,m}\rightarrow U_{A,m-1}\rightarrow0$, and by Lemma \ref{lemma of extension equals 0 with length 2} we have  an isomorphism between $\Ext^p_{\cc^b(\ca)}(U_{B,n},U_{A,m})$ and $\Ext^{p-1}_{\cc^b(\ca)}(U_{B,n},U_{A,m-1})$ for any $p\geq 1$.  From this isomorphism one can easily deduce the results as required.
\end{proof}

The set of the isomorphism classes of $\ca$ is denoted by $\Iso(\ca)$, and $\widehat{A}$ denotes the corresponding element in the Grothendieck group $K_0(\ca)$ for any object $A\in\ca$.

Since $\ca$ is hereditary, we have the multiplicative Euler form as usual
$$\langle -, -\rangle: K_0(\ca)\times K_0(\ca)\rightarrow\mathbb{Q}^{\times}$$
determined by
$$\langle \widehat{A}, \widehat{B}\rangle=\frac{|\Hom_{\ca}(A, B)|}{|\Ext^1_{\ca}(A, B)|}, \,\,\, \forall A, B\in\mathcal{A},$$
which is a bilinear form on the Grothendieck group $K_0(\ca)$.

From Proposition \ref{proposition homological finiteness of stalk complex} one can easily see that the multiplicative Euler form for $\cc^b(\ca)$ is also well-defined, i.e.,
$$\langle -,-\rangle : K_0(\cc^b(\ca))\times K_0(\cc^b(\ca))\rightarrow \mathbb{Q}^\times,$$
determined by
 $$\langle [M], [N]\rangle =\prod_{p=0}^{+\infty} |\Ext^p_{\cc^b(\ca)}(M,N)|^{(-1)^p},\,\,\, \forall  [M], [N]\in\Iso(\cc^b(\ca)),$$
which is a bilinear form on the Grothendieck group $K_0(\cc^b(\ca))$.

\begin{remark}In the above it would not cause confusion that we use the same symbol $\langle -,-\rangle $ to denote two Euler forms, since for any $A, B\in \ca$
we have $\langle \widehat{A}, \widehat{B}\rangle=\langle [U_{A,n}], [U_{B,n}]\rangle$ for any $n\in\mathbb{Z}$.
\end{remark}
In the following, $\cc^b_{ac}(\ca)$ denotes the category of bounded acyclic complexes over $\ca$.

For any $K\in\cc^b_{ac}(\ca)$, we know that $K$ has the Ext-dimension less than or equal to 1 and so we have
$$\langle [M], [K]\rangle=\frac{|\Hom_{\cc^b(\ca)}(M, K)|}{|\Ext^1_{\cc^b(\ca)}(M, K)|}$$ and
$$\langle [K], [M]\rangle=\frac{|\Hom_{\cc^b(\ca)}(K, M)|}{|\Ext^1_{\cc^b(\ca)}(K, M)|}.$$

The following result has been given by Gorsky in \cite{Gor13} without proofs. For reader's convenience we give some proofs.
\begin{proposition}\label{proposition euler form of acyclic complexes}
For any $A,B\in\ca$ and $m,n\in \mathbb{Z}$, we have $\langle [K_{A,m}],[U_{B,n}]\rangle=\langle \widehat{A}, \widehat{B}\rangle^{\delta^n_{m-1}}$, $\langle [U_{B,n}], [K_{A,m}]\rangle=\langle \widehat{B}, \widehat{A}\rangle^{\delta^{n}_m}$ and $\langle [K_{A,m}],[K_{B,n}]\rangle=\langle \widehat{A}, \widehat{B}\rangle^{(\delta^n_m+\delta^{n}_{m-1})}$.
\end{proposition}

\begin{proof}
We just prove the first identity and the proofs of others are similar. It is easy to see that if $n\neq m-1$, then  $\Ext^p_{\cc^b(\ca)}(K_{A,m},U_{B,n})=0$ for any $p\geq 0$. When $n=m-1$, we have that $\Hom _{\cc^b(\ca)}(K_{A,m},U_{B,n})\cong\Hom_\ca(A,B)$ and $\Ext^1_{\cc^b(\ca)}(K_{A,m},U_{B,n})\cong\Ext^1_{\ca}(A,B)$, and so $\langle [K_{A,m}],[U_{B,n}]\rangle =\langle \widehat{A},\widehat{B}\rangle^{\delta^n_{m-1}}$.
\end{proof}

By Proposition \ref{proposition homological finiteness of stalk complex}, one can obtain the following properties of stalk complexes.
\begin{proposition}\label{proposition Euler form of stalk complexes}
For any $A,B\in\ca$ and  $m>n\in \mathbb{Z}$, we have $\langle [U_{A,m}],[U_{B,m}]\rangle=\langle \widehat{A}, \widehat{B}\rangle$, $\langle [U_{A,m}],[U_{B,n}]\rangle=1, \langle [U_{B,n}],[U_{A,m}]\rangle=\langle \widehat{B}, \widehat{A}\rangle^{(-1)^{(m-n)}}$.
\end{proposition}
\section{Modified Ringel-Hall Algebras}
Let $\varepsilon$ be an essentially small exact category, linear over a finite field.
Assume that $\varepsilon$ has finite morphism and extension spaces:
$$|\Hom_\varepsilon(A,B)|<\infty,\quad |\Ext^1_{\varepsilon}(A,B)|<\infty,\,\,\forall A,B\in\varepsilon.$$

Given objects $A,B,C\in\varepsilon$, define $\Ext_\varepsilon^1(A,C)_B\subseteq \Ext_\varepsilon^1(A,C)$ as the subset parameterizing extensions whose middle term is isomorphic to $B$. We define the Ringel-Hall algebra $\ch(\varepsilon)$ to be the $\Q$-vector space whose basis is formed by the isomorphism classes $[A]$ of objects $A$ of $\varepsilon$, with the multiplication
defined by
$$[A]\diamond [C]=\sum_{[B]\in \Iso(\varepsilon)}\frac{|\Ext^1_{\varepsilon}(A,C)_B|}{|\Hom_{\varepsilon}(A,C)|}[B].$$
It is well-known that
the algebra $\ch(\varepsilon)$ is associative and unital. And the unit of the Ringel-Hall algebra is $[0]$, where $0$ is the zero object of $\varepsilon$, see \cite{R0} and also \cite{Sch2,P,Br}.

\subsection{Modified Ringel-Hall Algebras}
Let $\ch(\cc^b(\ca))$ be the Ringel-Hall algebra of $\cc^b(\ca)$, i.e., for any $L,M\in \cc^b(\ca)$, the Hall product is defined to be the following sum:
\begin{equation*}
[L]\diamond [M]=\sum_{[X]\in \Iso(\cc^b(\ca))}\frac{|\Ext^1_{\cc^b(\ca)}(L,M)_X|}{|\Hom_{\cc^b(\ca)}(L,M)|}[X].
\end{equation*}

Let $\ch(\cc^b(\ca))/I$ be the quotient algebra, where $I$ is the  ideal of $\ch(\cc^b(\ca))$ generated by all differences $[L]-[K\oplus M]$, if there is a short exact sequence $K\rightarrowtail L\twoheadrightarrow M$ in $\cc^b(\ca)$ with $K$ acyclic. We also denote by $\diamond$ the induced multiplication in $\ch(\cc^b(\ca))/I$.

One can easily prove  the following lemma.

\begin{lemma}\label{lemma multiplcation in quotient algebra}
For any $K\in\cc^b_{ac}(\ca)$ and $M\in\cc^{b}(\ca)$, then
\begin{eqnarray*}
&& [M]\diamond [K]=\frac{1}{\langle [M],[K]\rangle}[M\oplus K],
\end{eqnarray*}
in $\ch(\cc^b(\ca))/I$. In particular, for any $K_1,K_2\in\cc^b_{ac}(\ca)$, we have
\begin{eqnarray*}
&&[K_1]\diamond [K_2]=\frac{1}{\langle [K_1],[K_2]\rangle}[K_1\oplus K_2]
\end{eqnarray*}
in $\ch(\cc^b(\ca))/I$.
\end{lemma}

We set $S$ to be the subset of $\ch(\cc^b(\ca))/I$ formed by all $q[K]$, where $q\in \mathbb{Q}^\times, K\in\cc_{ac}^b(\ca)$. It is trivial to check that $S$ is a multiplicatively closed subset with identity $[0]\in S$. And similar to \cite{LuP} one can see that the right localization of $\ch(\cc^b(\ca))/I$ with respect to $S$ exists, denoted  by $(\ch(\cc^b(\ca))/I)[S^{-1}]$. Here we also denote by $\diamond$ the multiplication in $(\ch(\cc^b(\ca)/I)[S^{-1}]$.

\begin{definition}\label{definition of the modified Ringel-hall algebra} We denote $\cm\ch(\ca):=(\ch(\cc^b(\ca)/I)[S^{-1}]$, called the modified Ringel-Hall algebra (of bounded complexes) of $\ca$.
\end{definition}

Consider the set $\Iso(\cc^b_{ac}(\ca))$ of isomorphism classes $[K]$ of bounded acyclic complexes and its quotient by the following set of relations:
$$\langle [K_2]=[K_1\oplus K_3]| K_1\rightarrowtail K_2\twoheadrightarrow K_3 \mbox{ is a short exact sequence}\rangle.$$
If we endow $\Iso(\cc^b_{ac}(\ca))$ with the addition given by direct sums, this quotient gives the \emph{Grothendieck monoid} $M_0(\cc^b_{ac}(\ca))$ of $\cc^b_{ac}(\ca)$.
Define the \emph{quantum affine space} of acyclic complexes $\A_{ac}({\ca})$ as the $\Q$-monoid algebra of the Grothendieck monoid $M_0(\cc^b_{ac}(\ca))$, with the multiplication twisted by the inverse of the Euler form, i.e., the product of classes of acyclic complexes $K_1,K_2\in\cc_{ac}^b({\ca})$ is defined as follows:
$$[K_1]\diamond[K_2]:=\frac{1}{\langle [K_1],[K_2]\rangle}[K_1\oplus K_2].$$

Define the \emph{quantum torus} of acyclic complexes $\T_{ac}({\ca})$ as the $\Q$-group algebra of $K_0(\cc_{ac}^b({\ca}))$, with the
multiplication twisted by the inverse of the Euler form as above.
Note that $\T_{ac}(\ca)$ is the right localization of $\A_{ac}(\ca)$ with respect to the set formed by all the classes of acyclic complexes.

Note that $\cm\ch(\ca)$ is a $\T_{ac}(\ca)$-bimodule with the bimodule structure induced by the Hall product. By Lemma \ref{lemma multiplcation in quotient algebra}, we have the following lemma and the proof is similar to that of Lemma 3.4 in \cite{LuP}.
\begin{lemma}\label{lemma hall multiplicatoin of acyclic complexes}
For any $K_1,K_2\in\cc^b_{ac}(\ca)$ and $M\in\cc^b(\ca)$, we have
\begin{eqnarray*}
&&[K]\diamond [M]=\frac{1}{\langle [K],[M]\rangle}[K\oplus M]=\frac{\langle [M],[K]\rangle}{\langle [K],[M]\rangle}[M]\diamond [K]\\
&&[K_1]^{-1}\diamond [K_2]^{-1}=\langle [K_2],[K_1]\rangle [K_1\oplus K_2]^{-1},\\
&&[K_1]^{-1}\diamond [K_2]=\frac{\langle [K_1], [K_2]\rangle}{\langle [K_2],[K_1]\rangle}[K_2]\diamond [K_1]^{-1}
\end{eqnarray*}
in $\cm\ch(\ca)$.
\end{lemma}

For $A,B\in \ca$ with $\widehat{A}=\widehat{B}$ in the Grothendieck group of $\ca$, for any $m\in\mathbb{Z}$, in $\cm\ch(\ca)$ we have $[K_{A, m}]=[K_{B, m}]$, denoted by $K_{\widehat{A}, m}$, and if $\alpha=\widehat{A}-\widehat{B}$ we define $K_{\alpha, m}=\frac{1}{\langle\alpha, \widehat{B}\rangle}[K_{A, m}]\diamond[K_{B, m}]^{-1}$. As the corresponding proof in \cite{LuP} we know that $K_{\alpha, m}$ is independent of the expression of $\alpha$ in the Grothendieck group of $\ca$.

\subsection{The Basis of Modified Ringel-Hall Algebras}In this subsection, analogous to \cite{LuP} we also give a basis of $\cm\ch(\ca)$.

We define $I'$ to be the linear subspace of $\ch(\cc^b(\ca))$ spanned by
$$\{[L]-[K\oplus M]|\ K\rightarrowtail L\twoheadrightarrow M \mbox{ is a short exact sequence in }\cc^b(\ca)\mbox{ with }K\mbox{ acyclic} \}.$$
And define a bimodule structure of the quotient space $\ch(\cc^b(\ca))/I'$ over $\A_{ac}(\ca)$ by the rule
\begin{equation}\label{definition of bimodule}
[K]\diamond[M]:=\frac{1}{\langle [K],[M]\rangle}[K\oplus M],\quad [M]\diamond[K]:=\frac{1}{\langle [M],[K]\rangle}[M\oplus K].
\end{equation}
Furthermore, we set
$$(\ch(\cc^b(\ca))/I')[S^{-1}]:= \T_{ac}(\ca)\otimes_{\A_{ac}(\ca)}(\ch(\cc^b(\ca))/I')\otimes_{\A_{ac}(\ca)} \T_{ac}(\ca),$$
which is a bimodule over the quantum torus $\T_{ac}(\ca)$.
\begin{lemma}\label{lemma coincide of the structure of two bimodule}
The bimodule structure of
$(\ch(\cc^b(\ca))/I')[S^{-1}]$ is induced by the Hall product.
\end{lemma}
The proof of Lemma \ref{lemma coincide of the structure of two bimodule} is similar to that of Lemma 3.5 in \cite{LuP}, and we omit it.

For any nonzero complex $M\in\cc^b(\ca)$,  $$M=\cdots\rightarrow 0\rightarrow M^l\rightarrow\cdots\rightarrow M^r\rightarrow 0\rightarrow\cdots,$$
where $M^l$ is the leftmost nonzero component and $M^r$ is the rightmost nonzero component, then the {\em width} of $M$ is defined to be $r-l+1$. And if $M$ is zero, then the width of $M$ is defined to be $0$.
\begin{lemma}\label{lemma decomposition of acyclic complex}
Let $K=(K^i,d^i)\in\cc^b(\ca)$ be an acyclic complex. If the rightmost and the leftmost nonzero components of $K$ are $K^r$ and $K^l$ respectively. Then in $(\ch(\cc^b(\ca))/I')[S^{-1}]$ we have
$$[K]=[K_{\Im(d^l),l+1}]\diamond[K_{\Im(d^{l+1}),l+2}]\diamond\cdots\diamond[K_{\Im(d^{r-1}),r}].$$

\end{lemma}
\begin{proof}
If $r=l$, it is clear. If $r>l$, set $K_{r-l+1}=K$. And then we have the following short exact sequence of acyclic complexes
$$0\rightarrow K_{\Im(d^l),l+1}\rightarrow K_{r-l+1}\rightarrow K_{r-l} \rightarrow 0,$$
where $K_{r-l}=\cdots\rightarrow 0\rightarrow K^{l+1}/\Im(d^l) \rightarrow K^{l+2}\rightarrow\cdots\rightarrow K^{r-1}\rightarrow K^r\rightarrow 0 \rightarrow\cdots$. It is not hard to see
$$\langle K_{\Im(d^l),l+1},K_{r-l}\rangle=1$$
by Proposition \ref{proposition euler form of acyclic complexes}.
Hence we have
$$[K]=\langle K_{\Im(d^l),l+1},K_{r-l}\rangle [K_{\Im(d^l),l+1}]\diamond[K_{r-1}]=[K_{\Im(d^l),l+1}]\diamond[K_{r-l}].$$
Using induction on the width of complexes we can get that
$$[K]=[K_{\Im(d^l),l+1}]\diamond[K_{\Im(d^{l+1}),l+2}]\diamond\cdots\diamond[K_{\Im(d^{r-1}),r}].$$
\end{proof}

Similar to Proposition 3.7 in \cite{LuP}, we have the following proposition and we omit the proof.
\begin{proposition}\label{proposition equivalences of some exact sequences}
Let $\cb$ be an abelian category. Given objects $$U=(U^i,d_{U}^i),V=(V^i,d_{V}^i),W=(W^i,d_{W}^i)\in\cc^b(\cb),$$
if there is a short exact sequence $0\rightarrow U\xrightarrow{h_1} V\xrightarrow{h_2} W\rightarrow0$, then the following statements are equivalent.

(i) $0\rightarrow \Im(d_U^i)\xrightarrow{t_1^i} \Im(d_V^i)\xrightarrow{t_2^i} \Im(d_W^i)\rightarrow 0$ is short exact for each $i\in\mathbb{Z}$.

(ii) $0\rightarrow \Ker(d_U^i)\xrightarrow{s_1^i} \Ker(d_V^i)\xrightarrow{s_2^i} \Ker(d_W^i)\rightarrow 0$ is short exact for each $i\in\mathbb{Z}$.

(iii) $0\rightarrow H^i(U)\xrightarrow{r_1^i} H^i(V)\xrightarrow{r_2^i} H^i(W)\rightarrow 0$ is short exact for each $i\in\mathbb{Z}$.

\noindent Here the morphisms are induced by $h_1$ and $h_2$, and $H^i(M)(i\in\mathbb{Z})$ denote the homologies of a complex $M\in\cc^b(\cb)$.

In particular, if $U$ or $W$ is acyclic,
then for each $i\in\mathbb{Z}$
$$\widehat{\Im(d_V^i)}=\widehat{\Im(d_U^i)}+ \widehat{\Im(d_W^i)}$$
in $K_0(\cb)$.
\end{proposition}

\begin{proposition}\label{proposition multiplication decompisition of complexes}
Let $M=(M^i,d^i)\in\cc^b(\ca)$. If the rightmost and the leftmost nonzero components of $M$ are $M^r$ and $M^l$ respectively. Then in $(\ch(\cc^b(\ca))/I')[S^{-1}]$ we have
\begin{eqnarray*}
[M]&=&\langle \widehat{\Im(d^{r-1})},\widehat{\Ker(d^{r-1})}\rangle\langle \widehat{\Im(d^{r-2})}, \widehat{\Ker(d^{r-2})}\rangle\cdots
\langle \widehat{\Im(d^{l})},\widehat{\Ker(d^{l})}\rangle\\
&&[K_{\Im(d^{r-1}),r}]\diamond[K_{\Im(d^{r-2}),r-1}]\diamond\cdots\diamond[K_{\Im(d^{l}),l+1}]\diamond\\
&&[U_{H^r(M),r}\oplus U_{H^{r-1}(M),r-1}\oplus\cdots\oplus U_{H^{l}(M),l}].
\end{eqnarray*}

\end{proposition}

\begin{proof}
Firstly, we define the twisted bimodule structure over the quantum torus $\T_{ac}(\ca)$ by setting
$$[K]*[M]:=\langle [K],[M]\rangle[K]\diamond[M],\quad [M]*[K]:=\langle [M],[K]\rangle[M]\diamond[K].$$
So $[K]*[M]=[K\oplus M]=[M]*[K]$. And define $K_{\alpha, m}=[K_{A, m}]*[K_{B, m}]^{-1}$, if $\alpha=\widehat{A}-\widehat{B}$.

Let $N=\bigoplus_{i=l}^{r} U_{H^{i}(M),i}$. Then the complex $M$ is isomorphic to $N$ in the derived category $D^b(\ca))$ since $\ca$ hereditary. So it is not hard to see that there exists a bounded complex $X$ and acyclic bounded complexes $K_{i}(i=1,2,3,4)$ such that we have the following two short exact sequences
$$0\rightarrow K_1\rightarrow X\oplus K_2\rightarrow M\rightarrow 0$$
and $$0\rightarrow K_3\rightarrow X\oplus K_4\rightarrow N\rightarrow 0.$$
Thus $[X\oplus K_2\oplus K_4]=[K_1\oplus K_4\oplus M]$ and $[X\oplus K_2\oplus K_4]=[K_2\oplus K_3\oplus N]$, and then we get that
$$[K_1\oplus K_4]*[M]=[K_2\oplus K_3]*[N],$$i.e., $[M]=[K_2\oplus K_3]*[K_1\oplus K_4]^{-1}*[N].$ From above two exact sequences and by Proposition \ref{proposition equivalences of some exact sequences}, for each $j\in\mathbb{Z}$ we have that

$$\widehat{\Im(d^j)}=(\widehat{\Im(d_{K_2}^j)}+\widehat{\Im(d_{K_3}^j)})-(\widehat{\Im(d_{K_1}^j)}
+\widehat{\Im(d_{K_4}^j)}),$$
where $d_{K_i}$ denotes the differential of the complex $K_{i}$, $i=1,2,3,4$.
Therefore, by Lemma \ref{lemma decomposition of acyclic complex} we have
\begin{eqnarray*}
[M]&=&[K_{\Im(d^l),l+1}]*[K_{\Im(d^{l+1}),l+2}]*\cdots*[K_{\Im(d^{r-1}),r}]*[N]\\
&=&[K_{\Im(d^{r-1}),r}]*[K_{\Im(d^{r-2}),r-1}]*\cdots*[K_{\Im(d^{l}),1+1}]*[N].
\end{eqnarray*}

Following Proposition \ref{proposition euler form of acyclic complexes} and Proposition
\ref{proposition Euler form of stalk complexes}, we have
\begin{eqnarray*}
[M]&=&\langle[K_{\Im(d^{r-1}),r}], [K_{\Im(d^{r-2}),r-1}]\rangle\langle [K_{\Im(d^{r-1}),r}], [U_{H^{r-1}(M),r-1}]\rangle\\
&&\langle[K_{\Im(d^{r-2}),r-1}], [K_{\Im(d^{r-3}),r-2}]\rangle\langle [K_{\Im(d^{r-2}),r-1}], [U_{H^{r-2}(M),r-2}]\rangle\cdots\\
&&\langle[K_{\Im(d^{l+1}),l+2}], [K_{\Im(d^l),1+1}]\rangle\langle [K_{\Im(d^{l+1}),l+2}], [U_{H^{l+1}(M),l+1}]\rangle\\
&&\langle[K_{\Im(d^l),l+1}], [U_{H^{l}(M),l}]\rangle[K_{\Im(d^{r-1}),r}]\diamond[K_{\Im(d^{r-2}),r-1}]\diamond\cdots\diamond[K_{\Im(d^{l}),l+1}]\diamond[N]\\
&=&\langle\widehat{\Im(d^{r-1})}, \widehat{\Im(d^{r-2})}\rangle\langle\widehat{\Im(d^{r-1})}, \widehat{H^{r-1}(M)}\rangle\langle\widehat{\Im(d^{r-2})}, \widehat{\Im(d^{r-3})}\rangle\langle\widehat{\Im(d^{r-2})}, \widehat{H^{r-2}(M)}\rangle\\
&&\cdots\langle\widehat{\Im(d^{l+1})}, \widehat{\Im(d^l)}\rangle\langle\widehat{\Im(d^{l+1})}, \widehat{H^{l+1}(M)}\rangle\langle\widehat{\Im(d^l)}, \widehat{H^l(M)}\rangle\\
&&[K_{\Im(d^{r-1}),r}]\diamond[K_{\Im(d^{r-2}),r-1}]\diamond\cdots\diamond[K_{\Im(d^{l}),l+1}]\diamond[N]\\
&=&\langle \widehat{\Im(d^{r-1})},\widehat{\Ker(d^{r-1})}\rangle\langle \widehat{\Im(d^{r-2})}, \widehat{\Ker(d^{r-2})}\rangle\cdots\langle \widehat{\Im(d^{l+1})}, \widehat{\Ker(d^{l+1})}\rangle\langle \widehat{\Im(d^{l})},\widehat{\Ker(d^{l})}\rangle\\
&&[K_{\Im(d^{r-1}),r}]\diamond[K_{\Im(d^{r-2}),r-1}]\diamond\cdots\diamond[K_{\Im(d^{l}),l+1}]\diamond[N]
\end{eqnarray*}
\end{proof}

Similar to Theorem 3.9  and its proof in \cite{LuP} we can get a basis of  $(\ch(\cc^b(\ca))/I')[S^{-1}]$ as follows.
\begin{proposition}\label{proposition basis of vector space cor to modified}
$(\ch(\cc^b(\ca))/I')[S^{-1}]$ has a basis consisting of elements
$$[K_{\alpha_{r-1},r}]\diamond[K_{\alpha_{r-2},r-1}]
\diamond\cdots\diamond[K_{\alpha_{l},l+1}]
\diamond[U_{A_r,r}\oplus U_{A_{r-1},r-1}\oplus\cdots\oplus U_{A_{l},l}],$$
where $r,l\in\mathbb{Z}$, $r\geq l$, $\alpha _i\in K_0(\ca)$ and $A_{j}\in\Iso(\ca)$ for $l\leq i\leq r-1$ and $l\leq j\leq r$.
\end{proposition}

By the same proof as that of Theorem 3.12 in \cite{LuP}, we can get the following proposition.
\begin{proposition}\label{proposition iso quotient space and modified}
The natural projection $\ch(\cc^b(\ca))\rightarrow (\ch(\cc^b(\ca))/I')[S^{-1}]$ induces that
$\cm\ch(\ca)$ is isomorphic to $(\ch(\cc^b(\ca))/I')[S^{-1}]$ as $\T_{ac}(\ca)$-bimodules.
\end{proposition}

For the modified Ringel-Hall algebra we have the following consequence.
\begin{theorem}\label{theorem basis of modified}
$(1)$ Let $M=(M^i,d^i)\in\cc^b(\ca)$. If the rightmost and the leftmost nonzero components of $M$ are $M^r$ and $M^l$ respectively. Then in $\cm\ch(\ca)$ we have
\begin{eqnarray*}
[M]&=&\langle \widehat{\Im(d^{r-1})},\widehat{\Ker(d^{r-1})}\rangle\langle \widehat{\Im(d^{r-2})}, \widehat{\Ker(d^{r-2})}\rangle\cdots
\langle \widehat{\Im(d^{l})},\widehat{\Ker(d^{l})}\rangle\\
&&[K_{\Im(d^{r-1}),r}]\diamond[K_{\Im(d^{r-2}),r-1}]\diamond\cdots\diamond[K_{\Im(d^{l}),l+1}]\diamond\\
&&[U_{H^r(M),r}]\diamond [U_{H^{r-1}(M),r-1}]\diamond\cdots\diamond[U_{H^{l}(M),l}].
\end{eqnarray*}

$(2)$
$\cm\ch(\ca)$ has a basis consisting of elements
$$[K_{\alpha_{r-1},r}]\diamond[K_{\alpha_{r-2},r-1}]
\diamond\cdots\diamond[K_{\alpha_{l},l+1}]
\diamond[U_{A_r,r}]\diamond[U_{A_{r-1},r-1}]\diamond\cdots\diamond[U_{A_{l},l}],$$
where $r,l\in\mathbb{Z}$, $r\geq l$, $\alpha _i\in K_0(\ca)$ and $A_{j}\in\Iso(\ca)$ for $l\leq i\leq r-1$ and $l\leq j\leq r$.
\end{theorem}
\begin{proof}

For any objects $X, Y\in\cc^b(\ca)$ and $s,t\in\mathbb{Z}$ such that $s>t$, if the components $X^i=0$ and $Y^j=0$ for all $i<s$ and $j>t$, then one can easily see that $\mathrm{Hom}_{\cc^b(\ca)}(X, Y)=0$ and $\mathrm{Ext}^1_{\cc^b(\ca)}(X, Y)=0$, and so by the definition of modified Ringel-Hall algebra, we have  $[X]\diamond[Y]=[X\oplus Y]$. As a result, we have in $\cm\ch(\ca)$
$$[U_{H^r(M),r}\oplus U_{H^{r-1}(M),r-1}\oplus\cdots\oplus U_{H^{l}(M),l}]=[U_{H^r(M),r}]\diamond [U_{H^{r-1}(M),r-1}]\diamond\cdots\diamond[U_{H^{l}(M),l}].$$
Therefore (1) and (2) are obtained immediately from Proposition \ref{proposition multiplication decompisition of complexes},  Proposition \ref{proposition basis of vector space cor to modified} and Proposition \ref{proposition iso quotient space and modified}.
\end{proof}
\subsection{The generators and relations of modified Ringel-Hall algebras}In the following, we  describe the modified Ringel-Hall algebra by its generators and relations.

Given four objects $A, B, M$ and $N$ of $\ca$, let $V(M, B, A, N)$ be the subset of $\Hom(M,B)\times\Hom(B,A)\times\Hom(A,N)$ consisting of exact sequences $0\rightarrow M \rightarrow B\rightarrow A\rightarrow N\rightarrow 0$. The set $V(M, B, A, N)$ is finite and we define a rational number
$$\gamma_{AB}^{MN}:=\frac{|V(M, B, A, N)|}{a_Aa_B},$$
where $a_A=|\aut(A)|$ and $a_B=|\aut(B)|$.
\begin{lemma}\label{lemma multiplication of stalk complex }
Let $A, B\in\Iso(\ca)$, $n\in\mathbb{Z}$. In $\cm\ch(\ca)$, we have
$$[U_{B,n}]\diamond[U_{A,n+1}]=\sum_{M,N\in\Iso(\mathcal{A})}\gamma_{AB}^{MN}\frac{a_Aa_B}{a_Ma_N}\langle \widehat{B}-\widehat{M},\widehat{M}\rangle[K_{\widehat{B}-\widehat{M},n+1}]\diamond[U_{N,n+1}]\diamond[U_{M,n}].$$
\end{lemma}
\begin{proof}
Firstly, we claim that for the given objects $M, B, A, N\in\ca$, the set $$S=\{g\in\Hom(B, A)|\ \Ker(g)\cong M, \Coker(g)\cong N\}$$ is isomorphic to the set $$S'=\bigsqcup_{{\scriptsize\begin{array}{c}L\in\Iso(C^b(\ca))\\H^n(L)\cong M\\ H^{n+1}(L)\cong N\end{array}}}\Ext^1_{\cc^b(\ca)}(U_{B,n},U_{A,n+1})_{L}.$$
In fact, for any $\xi\in\Ext^1_{\cc^b(\ca)}(U_{B,n},U_{A,n+1})$, one can write  $\xi$ in the form $$0\rightarrow U_{A,n+1}\xrightarrow{\alpha} L\xrightarrow{\beta} U_{B,n}\rightarrow 0,$$
where $L=\cdots\rightarrow 0\rightarrow B\xrightarrow{g}A\rightarrow 0\rightarrow\cdots$ with $B$ sitting in the degree $n$ and $A$ in the degree $n+1$, $\alpha=(\cdots,0, 1_A,0,\cdots), \beta=(\cdots,0,1_B,0,\cdots)$. Therefore, for any $\xi,\xi'\in\Ext^1_{\cc^b(\ca)}(U_{B,n},U_{A,n+1})$, assume that $\xi=0\rightarrow U_{A,n+1}\xrightarrow{\alpha} L\xrightarrow{\beta} U_{B,n}\rightarrow 0$ and $\xi'=0\rightarrow U_{A,n+1}\xrightarrow{\alpha} L'\xrightarrow{\beta} U_{B,n}\rightarrow 0$, where $L=\cdots\rightarrow 0\rightarrow B\xrightarrow{g}A\rightarrow 0\rightarrow\cdots$ and $L'=\cdots\rightarrow 0\rightarrow B\xrightarrow{g'}A\rightarrow 0\rightarrow\cdots$. Clearly we have $[\xi]=[\xi']$ if and only if $g=g'$. Thus $S$ is isomorphic to $S'$.

It is easy to see that $|S|=\frac{|V(M,B,A,N)|}{a_Ma_N}$ and hence
$$\sum_{\scriptsize{\begin{array}{c}L\in\Iso(C^b(\ca))\\H^n(L)\cong M\\ H^{n+1}(L)\cong N\end{array}}}|\Ext^1_{\cc^b(\ca)}(U_{B,n},U_{A,n+1})_{L}|=\frac{|V(M,B,A,N)|}{a_Ma_N}$$

By Proposition \ref{proposition multiplication decompisition of complexes}, in $\cm\ch(\ca)$ we  have

$$[L]=\langle\widehat{B}-\widehat{M},\widehat{M}\rangle[K_{\widehat{B}-\widehat{M},n+1}]\diamond[U_{N,n+1}]\diamond[ U_{M,n}],$$
provided that $L=\cdots\rightarrow 0\rightarrow B\xrightarrow{g}A\rightarrow 0\rightarrow\cdots$, with $B$ sitting in the degree $n$ and $A$ in the degree $n+1$, satisfies $H^n(L)\cong M, H^{n+1}(L)\cong N$.

Note that $\Hom_{\cc^b(\ca)}(U_{B,n},U_{A,n+1})=0$. Therefore, in $\cm\ch(\ca)$ we have
\begin{eqnarray*}
& &[U_{B,n}]\diamond[U_{A,n+1}]\\
&=&\sum_{L\in\Iso(\cc^b(\ca))}|\Ext^1_{\cc^b(\ca)}(U_{B,n},U_{A,n+1})_{L}|[L]\\
&=&\sum_{M,N\in\Iso(\ca)}\left(\sum_{\scriptsize{
\begin{array}{c}L\in\Iso(C^b(\ca))\\H^n(L)\cong M\\ H^{n+1}(L)\cong N\end{array}}}|\Ext^1_{\cc^b(\ca)}(U_{B,n},U_{A,n+1})_{L}|\right)\\
& & \langle \widehat{B}-\widehat{M},\widehat{M}\rangle[K_{\widehat{B}-\widehat{M},n+1}]\diamond[U_{N,n+1}]\diamond[U_{M,n}]\\
&=&\sum\limits_{M,N\in\Iso(\mathcal{A})}\gamma_{AB}^{MN}\frac{a_Aa_B}{a_Ma_N}\langle \widehat{B}-\widehat{M},\widehat{M}\rangle[K_{\widehat{B}-\widehat{M},n+1}]\diamond[U_{N,n+1}]\diamond[U_{M,n}].
\end{eqnarray*}
\end{proof}

\begin{proposition}\label{proposition relations}
The modified Ringel-Hall algebra $\cm\ch(\ca)$ is generated by the set
$$\{U_{A,n},K_{\alpha,n}|\ A\in\Iso(\ca), \alpha\in K_0(\ca), n\in\mathbb{Z}\}$$ with the defining relations (\ref{relation in modified 1}) - (\ref{relation in modified 10}) as follows, where we write $U_{A,n}:=[U_{A,n}]$.
\begin{eqnarray}
U_{A,n}\diamond U_{B,n}&=&\sum\limits_{C\in\Iso(\ca)}\frac{|\Ext^1_{\ca}(A,B)_C|}{|\Hom_{\ca}(A,B)|}U_{C,n},\label{relation in modified 1}\\
K_{\alpha,n}\diamond U_{A,n}&=&\langle\widehat{A},\alpha\rangle U_{A,n}\diamond K_{\alpha,n},\\
K_{\alpha, n}\diamond K_{\beta, n}&=&\frac{1}{\langle\alpha, \beta\rangle} K_{\alpha+\beta, n},
\end{eqnarray}
\begin{eqnarray}
U_{A,n}\diamond K_{\alpha,n+1}&=&\langle\alpha, \widehat{A}\rangle K_{\alpha,n+1}\diamond U_{A,n},\\
K_{\alpha,n}\diamond U_{A,n+1}&=&U_{A,n+1}\diamond K_{\alpha,n},\\
K_{\alpha,n}\diamond K_{\beta,n+1}&=&\langle\beta, \alpha\rangle K_{\beta,n+1}\diamond K_{\alpha,n},
\end{eqnarray}
\begin{eqnarray}
&U_{B,n}\diamond U_{A,n+1}=\sum\limits_{M,N\in\Iso(\ca)}\gamma_{AB}^{MN}\frac{a_Aa_B}{a_Ma_N}\langle\widehat{B}-\widehat{M}, \widehat{M}\rangle K_{\widehat{B}-\widehat{M},n+1}\diamond U_{N,n+1}\diamond U_{M,n},
\end{eqnarray}
and if $|m-n|\geq 2$, then
\begin{eqnarray}
U_{A,m}\diamond U_{B,n} &=& U_{B,n}\diamond U_{A,m},\\
K_{\alpha,m}\diamond U_{B,n} &=&U_{B,n}\diamond K_{\alpha,m},\\
K_{\alpha,m}\diamond K_{\beta,n}&=&K_{\beta,n}\diamond K_{\alpha,m}.\label{relation in modified 10}
\end{eqnarray}
\end{proposition}

\begin{proof}
By Theorem \ref{theorem basis of modified} we know that the set $\{U_{A,n},K_{\alpha,n}|A\in\Iso(\ca), \alpha\in K_0(\ca), n\in\mathbb{Z}\}$ is a generating set of $\cm\ch(\ca)$. From the definition of modified Ringel-Hall algebra and Lemma \ref{lemma multiplication of stalk complex }
 one can easily get that these generators satisfy the relations (2)-(11).  Using the basis of $\cm\ch(\ca)$ in Theorem \ref{theorem basis of modified} one can see that these relations are the defining relations.
\end{proof}
\section{A new proof of  Green's formula}

For any objects $A, B, C\in\ca$, we use the symbol $g^C_{AB}$ to denote the number of subobjects $B'$ of $C$ such that $B'\cong B$ and $C/B'\cong A$, called a Hall number. Then one have the following homological formula (see \cite{Rie,P}) $$g^C_{AB}=\frac{|\Ext_{\ca}^1(A, B)_C|}{|\Hom_{\ca}(A, B)|}\frac{a_C}{a_Aa_B}.$$

The following is  Green's formula.
\begin{theorem}[\cite{Gr},Theorem 2]
Let $A, B, A', B'$ be fixed objects of $\ca$. Then there holds
\begin{eqnarray*}
&&a_Aa_Ba_{A'}a_{B'}\sum\limits_{C\in\Iso(\ca)}g_{AB}^{C}g_{A'B'}^{C}\frac{1}{a_C}\\
&=&\sum\limits_{X,Y,X',Y'\in\Iso(\ca)}\frac{|\Ext_{\ca}^1(X,Y')|}{|\Hom_{\ca}(X,Y')|}
g_{XX'}^{A}g_{YY'}^{B}g_{XY}^{A'}g_{X'Y'}^{B'}a_Xa_Ya_{X'}a_{Y'}.
\end{eqnarray*}
\end{theorem}

The known proofs including  the original one (see  \cite{Gr}, \cite{R5}, \cite{Sch2}) were to notice that in  Green's formula the both hand sides are related to the cardinal numbers respectively of the two sets, where one set consists of  some crosses determined by two short exact sequences and the other one consists of  some squares determined by four short exact sequences, and to find a bijection of the two sets. These proofs are fairly straightforward, however they are somewhat complicated to read. In the following we give a new proof  by using the associative multiplication of the modified Ringel-Hall algebra.

\begin{proof}
For fixed objects $A,B,A'\in\ca$ and any $n\in\mathbb{Z}$, we have
\begin{eqnarray*}
&&(U_{A,n}\diamond U_{B,n})\diamond U_{A',n+1} \\
&=&\sum\limits_{C\in\Iso(\ca)}g_{AB}^C\frac{a_Aa_B}{a_C}U_{C,n}\diamond U_{A',n+1}\\
&=&\sum\limits_{C,B',N\in\Iso(\ca)}g_{AB}^C\gamma_{A'C}^{B'N}\langle \widehat{A'}-\widehat{N},\widehat{B'}\rangle\frac{a_Aa_Ba_{A'}}{a_{B'}a_N}K_{\widehat{A'}-\widehat{N},n+1}\diamond U_{N,n+1}\diamond U_{B',n}\\
&=&\sum\limits_{B',N\in\Iso(\ca)}\left(\sum\limits_{C\in\Iso(\ca)}g_{AB}^C\gamma_{A'C}^{B'N}\langle \widehat{A'}-\widehat{N},\widehat{B'}\rangle\right)\frac{a_Aa_Ba_{A'}}{a_{B'}a_N}K_{\widehat{A'}-\widehat{N},n+1}\diamond U_{N,n+1}\diamond U_{B',n}.
\end{eqnarray*}
On the other hand, we have
\begin{eqnarray*}
&&U_{A,n}\diamond(U_{B,n}\diamond U_{A',n+1}) \\
&=&\sum\limits_{X,Y'\in\Iso(\ca)}\gamma_{A'B}^{Y'X}\frac{a_{A'}a_B}{a_Xa_{Y'}}\langle\widehat{B}-\widehat{Y'},\widehat{Y'}\rangle U_{A,n}\diamond K_{\widehat{B}-\widehat{Y'},n+1}\diamond U_{X,n+1}\diamond U_{Y',n}
\end{eqnarray*}
\begin{eqnarray*}
&=&\sum\limits_{X,Y'\in\Iso(\ca)}\gamma_{A'B}^{Y'X}\frac{a_{A'}a_B}{a_Xa_{Y'}}\langle\widehat{B}-\widehat{Y'},\widehat{Y'}+\widehat{A}\rangle  K_{\widehat{B}-\widehat{Y'},n+1}\diamond U_{A,n}\diamond U_{X,n+1}\diamond U_{Y',n}\\
&=&\sum\limits_{X,Y',N,X'\in\Iso(\ca)}\gamma_{A'B}^{Y'X}\gamma_{XA}^{X'N}\frac{a_Aa_{A'}a_B}{a_{Y'}a_{X'}a_N}\langle\widehat{B}-\widehat{Y'},\widehat{Y'}+\widehat{A}\rangle\langle \widehat{A}-\widehat{X'},\widehat{X'}\rangle\\
 &&\ \ \ K_{\widehat{B}-\widehat{Y'},n+1}\diamond K_{\widehat{A}-\widehat{X'},n+1}\diamond U_{N,n+1}\diamond U_{X',n}\diamond U_{Y',n}\\
&=&\sum\limits_{X,Y',N,X'\in\Iso(\ca)}\gamma_{A'B}^{Y'X}\gamma_{XA}^{X'N}\frac{a_Aa_{A'}a_B}{a_{Y'}a_{X'}a_N}\frac{\langle \widehat{A}-\widehat{X'},\widehat{X'}\rangle}{\langle\widehat{Y'}-\widehat{B},\widehat{Y'}+\widehat{X'}\rangle}\\
 &&\ \ \ K_{\widehat{A'}-\widehat{N},n+1}\diamond U_{N,n+1}\diamond U_{X',n}\diamond U_{Y',n}\\
&=&\sum\limits_{X,Y',N,X',B'\in\Iso(\ca)}\gamma_{A'B}^{Y'X}\gamma_{XA}^{X'N}g_{X'Y'}^{B'}\frac{a_Aa_{A'}a_B}{a_Na_{B'}}\frac{\langle \widehat{A}-\widehat{X'},\widehat{X'}\rangle}{\langle\widehat{Y'}-\widehat{B},\widehat{Y'}+\widehat{X'}\rangle} \\
  &&\ \ \ K_{\widehat{A'}-\widehat{N},n+1}\diamond U_{N,n+1}\diamond U_{B',n}\\
&=&\sum\limits_{B',N\in\Iso(\ca)}\left(\sum\limits_{X,Y',X'\in\Iso(\ca)}\gamma_{A'B}^{Y'X}\gamma_{XA}^{X'N}g_{X'Y'}^{B'}\frac{\langle \widehat{A}-\widehat{X'},\widehat{X'}\rangle}{\langle\widehat{Y'}-\widehat{B},\widehat{Y'}+\widehat{X'}\rangle}\right)\frac{a_Aa_{A'}a_B}{a_Na_{B'}}
\\
 &&\ \ \ K_{\widehat{A'}-\widehat{N},n+1}\diamond U_{N,n+1}\diamond U_{B',n}.
\end{eqnarray*}
From the basis of the twisted modified Ringel-Hall algebra deduced from the Theorem \ref{theorem basis of modified}(2), for fixed isomorphism classes $B',N\in\Iso(\ca)$, we have
\begin{eqnarray}
\sum\limits_{C\in\Iso(\ca)}g_{AB}^C\gamma_{A'C}^{B'N}\langle\widehat{A'}-\widehat{N},\widehat{B'}\rangle=\sum\limits_{X,Y',X'\in\Iso(\ca)}\gamma_{A'B}^{Y'X}\gamma_{XA}^{X'N}g_{X'Y'}^{B'}\frac{\langle \widehat{A}-\widehat{X'},\widehat{X'}\rangle}{\langle\widehat{Y'}-\widehat{B},\widehat{Y'}+\widehat{X'}\rangle}.\label{equ coefficient}
\end{eqnarray}

By the definitions, for any $D, E, F, G\in\ca$ one can easily get that  $$\gamma_{DE}^{FG}=\sum\limits_{I\in\Iso(\ca)}g_{IF}^Eg_{GI}^D\frac{a_Fa_Ia_G}{a_Da_E}.$$
In particular, we have $\gamma_{DE}^{F0}=g_{DF}^E\frac{a_F}{a_E}.$

In \text{the identity}~(\ref{equ coefficient}) if we set $N=0$, then
\begin{eqnarray*}
\text{LHS of the identity}~(\ref{equ coefficient})
&=&\sum\limits_{C\in\Iso(\ca)}g_{AB}^{C}g_{A'B'}^{C}\frac{a_{B'}}{a_{C}}\langle\widehat{A'},\widehat{B'}\rangle\\
&=&\langle\widehat{A'},\widehat{B'}\rangle\sum\limits_{C\in\Iso(\ca)}g_{AB}^{C}g_{A'B'}^{C}\frac{a_{B'}}{a_{C}},
\end{eqnarray*}
and
\begin{eqnarray*}
\text{RHS of the identity}~(\ref{equ coefficient})&=&\sum\limits_{X,Y',X'\in\Iso(\ca)}\gamma_{A'B}^{Y'X}\gamma_{XA}^{X'0}g_{X'Y'}^{B'}\frac{\langle \widehat{A}-\widehat{X'},\widehat{X'}\rangle}{\langle\widehat{Y'}-\widehat{B},\widehat{Y'}+\widehat{X'}\rangle}\\
&=&\sum\limits_{X,Y,X',Y'\in\Iso(\ca)}g_{YY'}^Bg_{XY}^{A'}g_{X'Y'}^{B'}g_{XX'}^{A}\frac{a_Xa_Ya_{Y'}a_{X'}}{a_Aa_{A'}a_B}\frac{\langle \widehat{A}-\widehat{X'},\widehat{X'}\rangle}{\langle\widehat{Y'}-\widehat{B},\widehat{Y'}+\widehat{X'}\rangle}.
\end{eqnarray*}
Note that one only need to consider the non-zero terms in the above formula and so we can set $\widehat{B}=\widehat{Y}+\widehat{Y'}, \,
\widehat{A'}=\widehat{X}+\widehat{Y}, \, \widehat{B'}=\widehat{X'}+\widehat{Y'} $ and $\widehat{A}=\widehat{X}+\widehat{X'}$.
Thus we have
\begin{eqnarray*}
\frac{\langle \widehat{A}-\widehat{X'},\widehat{X'}\rangle}{\langle\widehat{A'},\widehat{B'}\rangle\langle\widehat{Y'}-\widehat{B},\widehat{Y'}+\widehat{X'}\rangle}&=&\frac{\langle \widehat{Y},\widehat{B'}\rangle\langle\widehat{A}-\widehat{X'},\widehat{X'}\rangle}{\langle\widehat{A'},\widehat{B'}\rangle}\\
&=&\frac{\langle \widehat{X},\widehat{X'}\rangle}{\langle\widehat{X},\widehat{B'}\rangle}\\
&=&\frac{1}{\langle\widehat{X},\widehat{Y'}\rangle}\\
&=&\frac{|\Ext_{\ca}^1(X,Y')|}{|\Hom_{\ca}(X,Y')|}.
\end{eqnarray*}
Therefore we obtain Green's formula.

\end{proof}

\section{Derived Hall Algebras and  Modified Ringel-Hall Algebras}

\subsection{Twisted Derived Hall Algebras and Twisted Modified Ringel-Hall Algebras}

Let $\ct$ be a $k$-additive triangulated category with the translation $T=[1]$ satisfying
\begin{itemize}
\item[(i)] $\dim_k\Hom_{\ct}(X, Y)<\infty$ for any two objects $X$ and $Y$,
\item[(ii)] $\End_{\ct}(X)$ is local for any indecomposable object $X$,
\item[(iii)] $\ct$ is (left) locally finite; that is, $\sum_{i\geq 0}\dim_k\Hom_{\ct}(X[i], Y)<\infty$ for any $X$ and $Y$.
\end{itemize}

The derived Hall algebra $\cd\ch(\ct)$ of the triangulated category $\ct$ is the $\mathbb{Q}$-space with the basis $\{[X]|X\in\ct\}$ and the multiplication is defined by
$$[X][Y]=\sum_{[L]}\frac{|\Ext^1_\ct(X, Y)_L|}{\prod_{i\geq 0}|\Hom_{\ct}(X[i], Y)|^{(-1)^i}}[L],$$
where $\Ext^1_\ct(X, Y)_L$ is defined to be $\Hom_{\ct}(X, Y[1])_{L[1]}$ which denotes the subset of $\Hom(X, Y[1])$ consisting of morphisms $l: X\rightarrow Y[1]$ whose cone $\Cone(l)$ is isomorphic to $L[1]$. Here the definition we used is a version in \cite{XX2} so-called the Drinfeld dual of the derived Hall algebra given  by To\"{e}n in \cite{T} and also by Xiao-Xu in \cite{XX}.

Similar to the work of To\"{e}n we can describe the derived Hall algebra for the hereditary abelian category $\ca$ as follows.

\begin{proposition}\label{proposition derived hall algebra}
$\cd\ch(\ca)$ is an associative and unital $\mathbb{Q}$-algebra generated by the set $$\{Z_A^{[n]}|\ A\in\Iso(\ca), n\in\mathbb{Z}\},$$  with the defining relations as follows.
\begin{eqnarray}
Z_A^{[n]}Z_B^{[n]}&=&\sum\limits_{C\in \Iso(\ca)}\frac{|\Ext^1_\ca(A,B)_C|}{|\Hom_\ca(A,B)|}Z_C^{[n]},\\
Z_B^{[n]}Z_A^{[n+1]}&=&\sum\limits_{M, N\in\Iso(\ca)}\gamma_{AB}^{MN}\frac{a_Aa_B}{a_Ma_N}\frac{1}{\langle \widehat{N}, \widehat{M}\rangle}Z_N^{[n+1]}Z_M^{[n]},\\
Z_B^{[n]}Z_A^{[m]}&=&\langle \widehat{A}, \widehat{B}\rangle^{(-1)^{m-n}}Z_A^{[m]}Z_B^{[n]}\ \text{for~} m>n+1.
\end{eqnarray}
\end{proposition}

We can define the twisted derived Hall algebra $\cd\ch_{tw}(\ca)$ by the twisted multiplication
\begin{equation*}
[X]*[Y]=\prod_{i\in\mathbb{Z}}|\Hom_{D^b(\ca)}(X, Y[i])|^{(-1)^i}[X][Y],
\end{equation*}
for any $[X], [Y]\in\Iso(D^b(\ca))$. Here $\prod_{i\in\mathbb{Z}}|\Hom_{D^b(\ca)}(X, Y[i])|^{(-1)^i}$ is the Euler
form of the derived category $D^b(\ca)$.
Then one can easily get the following proposition.
\begin{proposition}
$\cd\ch_{tw}(\ca)$ is an associative and unital $\mathbb{Q}$-algebra generated by the set $$\{Z_A^{[n]}|\ A\in\Iso(\ca), n\in\mathbb{Z}\},$$
with the defining relations as follows.
\begin{eqnarray}
Z_A^{[n]}*Z_B^{[n]}&=&\sum\limits_{C\in \Iso(\ca)}\langle \widehat{A}, \widehat{B}\rangle \frac{|\Ext^1_\ca(A,B)_C|}{|\Hom_\ca(A,B)|}Z_C^{[n]},\label{the first relation in twisted derived hall algebra}\\
Z_B^{[n]}*Z_A^{[n+1]}&=&\sum\limits_{M, N\in\Iso(\ca)}\gamma_{AB}^{MN}\frac{a_Aa_B}{a_Ma_N}\frac{1}{\langle \widehat{B}, \widehat{A}\rangle}Z_N^{[n+1]}*Z_M^{[n]},\label{relation in twisted derived hall alg of adjacent 2 lattices}\\
Z_B^{[n]}*Z_A^{[m]}&=&\langle \widehat{B} , \widehat{A}\rangle^{(-1)^{n-m}}Z_A^{[m]}*Z_B^{[n]}\ \text{for~} m>n+1.\label{the last relation in twisted derived hall alg }
\end{eqnarray}
\end{proposition}

Now we define the twisted modified Ringel-Hall algebra $\cm\ch_{tw}(\ca)$ by the Euler form for $\cc^b(\ca)$, i.e., the multiplication in $\cm\ch_{tw}(\ca)$ is given by
\begin{equation}\label{twisted multiplication of modified}
[M_1]*[M_2]=\langle [M_1],[M_2]\rangle[M_1]\diamond[M_2],\forall [M_1],[M_2]\in\mathrm{Iso}(\cc^{b}(\ca)).
\end{equation}
Then $\cm\ch_{tw}(\ca)$ is still an associative and unital $\mathbb{Q}$-algebra.

For $A,B\in \ca$, If $\alpha=\widehat{A}-\widehat{B}$, we define $K_{\alpha, m}=[K_{A, m}]*[K_{B, m}]^{-1}$.
Let $\mathbb{T}_{ac}^{tw}(\ca)$ denote the twisted quantum torus of acyclic complexes. Then $\mathbb{T}_{ac}^{tw}(\ca)$ is the subalgebra of $\cm\ch_{tw}(\ca)$ generated by $\{ K_{\alpha, m} | \ \alpha\in K_0(\ca), m\in \mathbb{Z} \}$.  By the twisted multiplication one can easily see that
$\mathbb{T}_{ac}^{tw}(\ca)$ is commutative and so it is isomorphic to the group algebra of the Grothendieck group $K_0(\cc_{ac}^b\ca)$.

By Proposition \ref{proposition euler form of acyclic complexes}, Proposition \ref{proposition Euler form of stalk complexes} and Proposition \ref{proposition relations} one can easily get the following proposition.
\begin{proposition}\label{proposition of relations in twist MH(A)}
$\cm\ch_{tw}(\ca)$ is generated by the set
$$\{U_{A,n},K_{\alpha,n}|\ A\in\Iso(\ca), \alpha\in K_0(\ca), n\in\mathbb{Z}\}$$
with the defining relations as follows.
\begin{eqnarray}
U_{A,n}*U_{B,n}&=&\sum\limits_{C\in \Iso(\ca)}\langle \widehat{A},\widehat{B}\rangle\frac{|\Ext^1_{\ca}(A,B)_C|}{|\Hom_{\ca}(A,B)|}U_{C,n}, \\
K_{\alpha,n}*U_{A,n}&=&U_{A,n}*K_{\alpha,n},\\
K_{\alpha, n}*K_{\beta, n}&=&K_{\alpha+\beta, n};
\end{eqnarray}
\begin{eqnarray}
U_{A,n}*K_{\alpha,n+1}&=&K_{\alpha,n+1}*U_{A,n},\\
K_{\alpha,n}*U_{A,n+1}&=&U_{A,n+1}*K_{\alpha,n},\\
K_{\alpha,n}*K_{\beta,n+1}&=&K_{\beta,n+1}*K_{\alpha,n},
\end{eqnarray}
\begin{eqnarray}
\quad U_{B,n}*U_{A,n+1}&=&\sum\limits_{M,N\in\Iso(\mathcal{A})}\gamma_{AB}^{MN}\frac{a_Aa_B}{a_Ma_N}\frac{1}{\langle\widehat{B}, \widehat{A}\rangle}U_{N,n+1}*U_{M,n}*K_{\widehat{B}-\widehat{M},n+1};
\end{eqnarray}
and if $m>n+1$, then
\begin{eqnarray}
U_{B,n}*U_{A,m}&=&\langle \widehat{B} , \widehat{A}\rangle^{(-1)^{m-n}}U_{A,m}*U_{B,n},\label{the quasi-commutative of twisted modified hall}\\
U_{B,n}*K_{\alpha,m}=K_{\alpha,m}*U_{B,n},&&U_{B,m}*K_{\alpha,n}=K_{\alpha,n}*U_{B,m}\\
K_{\beta,n}*K_{\alpha,m}&=&K_{\alpha,m}*K_{\beta,n},
\end{eqnarray}
 for any $A, B\in\Iso(\ca)$, $\alpha, \beta\in K_{0}(\ca)$ and $m, n\in\mathbb{Z}$.
\end{proposition}

\begin{remark}
Note that any  $K_{\alpha,m}$ is commutative with all elements in $\cm\ch_{tw}(\ca)$. This is similar to the case of the twisted semi-derived Hall algebra given by  Gorsky in \cite{Gor13}.
\end{remark}

\subsection{Main results} The following is our embedding theorem.
\begin{theorem}\label{main result}
There is an embedding of the twisted derived Hall algebra $\cd\ch_{tw}(\ca)$ in the twisted modified Ringel-Hall algebra $\cm\ch_{tw}(\ca)$.
\end{theorem}
\begin{proof}
We construct a map $\iota: \cd\ch_{tw}(\ca)\rightarrow\cm\ch_{tw}(\ca)$ defined by
$$Z_A^{[0]}\mapsto U_{A,0},\quad Z_A^{[n]}\mapsto U_{A,n}*\prod_{i=1}^{n}(K_{\widehat{A},n-i+1})^{(-1)^i}\ \text{and}\ Z_A^{[-n]}\mapsto U_{A,-n}*\prod_{i=1}^{n}(K_{\widehat{A},i-n})^{(-1)^i},$$
for any $A\in\Iso(\ca), n>0$, where both $\prod_{i=1}^{n}(K_{\widehat{A},n-i+1})^{(-1)^i}$ and $\prod_{i=1}^{n}(K_{\widehat{A},i-n})^{(-1)^i}$ denote the twisted multiplications of acyclic complexes in $\cm\ch_{tw}(\ca)$.

To prove $\iota$ is a well-defined morphism of algebras we just need to check the corresponding relations (\ref{the first relation in twisted derived hall algebra})-(\ref{the last relation in twisted derived hall alg }) in $\cd\ch_{tw}(\ca)$ under $\iota$.

Because any $K_{\alpha,n}$ is commutative in $\cm\ch_{tw}(\ca)$, we can easily obtain the relation (\ref{the first relation in twisted derived hall algebra}).

Let $A, B\in\Iso(\ca)$. For any $n>1$, we have the following identities.

\begin{eqnarray*}
&&\iota(Z_B^{[-n]})*\iota(Z_A^{[-n+1]})\\
&=&U_{B,-n}*\prod_{i=1}^{n}(K_{\widehat{B},i-n})^{(-1)^i}*U_{A,-n+1}*\prod_{j=1}^{n-1}(K_{\widehat{A},j-n+1})^{(-1)^j}\\
&=&U_{B,-n}*U_{A,-n+1}*\prod_{i=1}^{n}(K_{\widehat{B},i-n})^{(-1)^i}*\prod_{i=2}^{n}(K_{\widehat{A},i-n})^{(-1)^{(i-1)}}\\
&=&U_{B,-n}*U_{A,-n+1}*(K_{\widehat{B},-n+1})^{-1}*\prod_{i=2}^{n}(K_{\widehat{B},i-n}*(K_{\widehat{A},i-n})^{-1})^{(-1)^i}\\
&=&\sum\limits_{M, N\in\Iso(\mathcal{A})}\gamma_{AB}^{MN}\frac{a_Aa_B}{a_Ma_N}\frac{1}{\langle\widehat{B}, \widehat{A}\rangle}U_{N,-n+1}*U_{M,-n}\\&&*K_{\widehat{B}-\widehat{M},-n+1}*(K_{\widehat{B},-n+1})^{-1}*\prod_{i=2}^{n}(K_{\widehat{B}-\widehat{A},i-n})^{(-1)^i}\\
&=&\sum\limits_{M, N\in\Iso(\mathcal{A})}\gamma_{AB}^{MN}\frac{a_Aa_B}{a_Ma_N}\frac{1}{\langle\widehat{B}, \widehat{A}\rangle}U_{N,-n+1}*U_{M,-n}\\
&&*(K_{\widehat{M},-n+1})^{-1}*\prod_{i=2}^{n}(K_{\widehat{M},i-n}*(K_{\widehat{N},i-n})^{-1})^{(-1)^i}\\
&=&\sum\limits_{M, N\in\Iso(\mathcal{A})}\gamma_{AB}^{MN}\frac{a_Aa_B}{a_Ma_N}\frac{1}{\langle \widehat{B}, \widehat{A}\rangle}U_{N,-n+1}*\prod_{j=1}^{n-1}(K_{\widehat{N},j-n+1})^{(-1)^j}*\\
&&U_{M,-n}*\prod_{i=1}^{n}(K_{\widehat{M},i-n})^{(-1)^i}\\
&=&\sum\limits_{M, N\in\Iso(\mathcal{A})}\gamma_{AB}^{MN}\frac{a_Aa_B}{a_Ma_N}\frac{1}{\langle \widehat{B}, \widehat{A}\rangle}\iota(Z_N^{[-n+1]})*\iota(Z_M^{[-n]}),
\end{eqnarray*}

\begin{eqnarray*}
&&\iota(Z_B^{[-1]})*\iota(Z_A^{[0]})\\
&=&U_{B,-1}*(K_{\widehat{B},0})^{-1}*U_{A,0}\\
&=&\sum\limits_{M, N\in\Iso(\mathcal{A})}\gamma_{AB}^{MN}\frac{a_Aa_B}{a_Ma_N}\frac{1}{\langle \widehat{B}, \widehat{A}\rangle}U_{N,0}*U_{M,-1}*K_{\widehat{B}-\widehat{M},0}*(K_{\widehat{B},0})^{-1}\\
&=&\sum\limits_{M, N\in\Iso(\mathcal{A})}\gamma_{AB}^{MN}\frac{a_Aa_B}{a_Ma_N}\frac{1}{\langle \widehat{B}, \widehat{A}\rangle}U_{N,0}*U_{M,-1}*(K_{\widehat{M},0})^{-1}\\
&=&\sum\limits_{M, N\in\Iso(\mathcal{A})}\gamma_{AB}^{MN}\frac{a_Aa_B}{a_Ma_N}\frac{1}{\langle \widehat{B}, \widehat{A}\rangle}\iota(Z_N^{[0]})*\iota(Z_M^{[-1]}),
\end{eqnarray*}

\begin{eqnarray*}
&&\iota(Z_B^{[0]})*\iota(Z_A^{[1]})\\
&=&U_{B,0}*U_{A,1}*(K_{\widehat{A},1})^{-1}\\
&=&\sum\limits_{M, N\in\Iso(\mathcal{A})}\gamma_{AB}^{MN}\frac{a_Aa_B}{a_Ma_N}\frac{1}{\langle \widehat{B}, \widehat{A}\rangle}U_{N,1}*U_{M,0}*K_{\widehat{B}-\widehat{M},1}*(K_{\widehat{A},1})^{-1}\\
&=&\sum\limits_{M, N\in\Iso(\mathcal{A})}\gamma_{AB}^{MN}\frac{a_Aa_B}{a_Ma_N}\frac{1}{\langle \widehat{B}, \widehat{A}\rangle}U_{N,1}*(K_{N,1})^{-1}*U_{M,0}\\
&=&\sum\limits_{M, N\in\Iso(\mathcal{A})}\gamma_{AB}^{MN}\frac{a_Aa_B}{a_Ma_N}\frac{1}{\langle \widehat{B}, \widehat{A}\rangle}\iota(Z_N^{[1]})*\iota(Z_M^{[0]}),
\end{eqnarray*}
and
\begin{eqnarray*}
&&\iota(Z_B^{[n]})*\iota(Z_A^{[n+1]})\\
&=&U_{B,n}*U_{A,n+1}*\prod_{i=1}^{n}(K_{\widehat{B},n+1-i})^{(-1)^i}*(K_{\widehat{A},n+1})^{-1}*\prod_{j=2}^{n+1}(K_{\widehat{A},n+2-j})^{(-1)^j}\\
&=&U_{B,n}*U_{A,n+1}*(K_{\widehat{A},n+1})^{-1}*\prod_{i=1}^{n}(K_{\widehat{B},n+1-i}*(K_{\widehat{A},n+1-i})^{-1})^{(-1)^i}\\
&=&\sum\limits_{M, N\in\Iso(\mathcal{A})}\gamma_{AB}^{MN}\frac{a_Aa_B}{a_Ma_N}\frac{1}{\langle\widehat{B}, \widehat{A}\rangle}U_{N,n+1}*U_{M,n}*K_{\widehat{B}-\widehat{M},n+1}*\\
&&(K_{\widehat{A},n+1})^{-1}*\prod_{i=1}^{n}(K_{\widehat{B}-\widehat{A},n+1-i})^{(-1)^i}\\
&=&\sum\limits_{M, N\in\Iso(\mathcal{A})}\gamma_{AB}^{MN}\frac{a_Aa_B}{a_Ma_N}\frac{1}{\langle\widehat{B}, \widehat{A}\rangle}U_{N,n+1}*U_{M,n}*(K_{\widehat{N},n+1})^{-1}*\\&&\prod_{i=1}^{n}(K_{\widehat{M}-\widehat{N},n+1-i})^{(-1)^i}\\
&=&\sum\limits_{M, N\in\Iso(\mathcal{A})}\gamma_{AB}^{MN}\frac{a_Aa_B}{a_Ma_N}\frac{1}{\langle \widehat{B}, \widehat{A}\rangle}U_{N,n+1}*\prod_{j=1}^{n+1}(K_{\widehat{N},n+2-j})^{(-1)^j}*\\
&&U_{M,n}*\prod_{i=1}^{n}(K_{\widehat{M},n+1-i})^{(-1)^i}
\end{eqnarray*}
\begin{eqnarray*}
&=&\sum\limits_{M, N\in\Iso(\mathcal{A})}\gamma_{AB}^{MN}\frac{a_Aa_B}{a_Ma_N}\frac{1}{\langle \widehat{B}, \widehat{A}\rangle}\iota(Z_N^{[n+1]})*\iota(Z_M^{[n]}).
\end{eqnarray*}
So the relation (\ref{relation in twisted derived hall alg of adjacent 2 lattices}) holds.

For the last  relation (\ref{the last relation in twisted derived hall alg }), we can check it similarly for $m>n+1$ in the following five cases: $n=0, m>1; n<-1, m=0; n<0, m>0; n<m<0$ and $0<n<m$. Each case is clear  since  the quantum torus $\mathbb{T}_{ac}^{tw}(\ca)$ is commutative in $\cm\ch_{tw}(\ca)$.

Note that we have a similar basis as in Theorem \ref{theorem basis of modified} by just replacing the multiplication $\diamond$ with $*$. From this basis one can easily get that $\iota$ is injective. This completes the proof.
\end{proof}

Now we consider  the tensor algebra $\cd\ch_{tw}(\ca)\otimes_{\mathbb{Q}}\mathbb{T}_{ac}^{tw}(\ca)$ as usual, i.e.,
the multiplication is defined as follows.
$$(X_1\otimes T_1)(X_2\otimes T_2)=(X_1*X_2)\otimes (T_1*T_2),$$
for any $X_1,X_2\in\cd\ch_{tw}(\ca), T_1, T_2\in\mathbb{T}_{ac}^{tw}(\ca)$. Then the following result shows that  the above embedding can be extended to an isomorphism.
\begin{corollary}\label{corollary iso DH_{tw}(A) and MH(A)in section 1}
The tensor algebra $\cd\ch_{tw}(\ca)\otimes_{\mathbb{Q}}\mathbb{T}_{ac}^{tw}(\ca)$ is isomorphic to the twisted modified Ringel-Hall algebra $\cm\ch_{tw}(\ca)$.
\end{corollary}

\begin{proof}
Clearly the embedding $\iota$ defined in the proof of Theorem \ref{main result} can be extended to a morphism
$$\widetilde{\iota}: \cd\ch_{tw}(\ca)\otimes_{\mathbb{Q}}\mathbb{T}_{ac}^{tw}(\ca)\rightarrow \cm\ch_{tw}(\ca),$$
by $\widetilde{\iota}(X\otimes T):= \iota(X)*T$, for any $X\in\cd\ch_{tw}(\ca), T\in\mathbb{T}_{ac}^{tw}(\ca)$. It is easy to check that $\widetilde{\iota}$ is an epimorphism of algebras.

From the basis of $\cm\ch(\ca)$ given in Theorem \ref{theorem basis of modified} we easily know that  $\widetilde{\iota}$ is a monomorphism. This completes the proof.
\end{proof}

The following further consequence shows that  the twisted modified Ringel-Hall algebra is invariant under derived equivalences.
\begin{corollary}\label{corollary derived-invariance}
Let  $\cb$ be also an essentially small finitary
hereditary abelian k-category. If there exists a derived equivalence
$$F: \cd^b(\ca)\rightarrow \cd^b(\cb),$$
then $\cm\ch_{tw}(\ca)$ and $\cm\ch_{tw}(\cb)$ are isomorphic.
\end{corollary}

\begin{proof}
By Corollary \ref{corollary iso DH_{tw}(A) and MH(A)in section 1}, it is sufficient to prove that there exists an isomorphism
$$F_*: \cd\ch_{tw}(\ca)\otimes_{\mathbb{Q}}\mathbb{T}_{ac}^{tw}(\ca)\rightarrow \cd\ch_{tw}(\cb)\otimes_{\mathbb{Q}}\mathbb{T}_{ac}^{tw}(\cb).$$

Clearly the equivalence $F$ can induce an isomorphism between the Grothendieck groups $K_0(\ca)$ and $K_0(\cb)$, still denoted by $F$. So we can get an isomorphism $F_*$ between $\mathbb{T}_{ac}^{tw}(\ca)$ and $\mathbb{T}_{ac}^{tw}(\cb)$ by setting $F_*(K_{\alpha, n})=K_{F(\alpha), n}$ for any $\alpha\in K_0(\ca)$ and $n\in \mathbb{Z}$.

Because the twisted derived Hall algebras is invariant under derived equivalences, we denote by $F_*$ the induced isomorphism between the twisted derived Hall algebras. For any object $X\otimes T$ in $\cd\ch_{tw}(\ca)\otimes_{\mathbb{Q}}\mathbb{T}_{ac}^{tw}(\ca)$, set $$F_*(X\otimes T)=F_*(X)\otimes F_*(T).$$
Clearly  this is an isomorphism of algebras. This finishes the proof.
\end{proof}

\begin{remark}
If $\ca$ has enough projectives, then as in \cite{LuP} one can prove that the modified Ringel-Hall algebra $\cm\ch(\ca)$ is isomorphic to the Bridgeland's Hall algebra in \cite{Br}  from  bounded complexes of projectives, which is also the semi-derived Hall algebra defined by Gorsky in \cite{Gor16} from the Frobenius category consisting of bounded complexes of projectives. In this case our results above are essentially same as those gotten by Gorsky in \cite{Gor16}.
\end{remark}

\end{document}